\documentclass{amsart}

\usepackage[english]{babel}
\usepackage[utf8]{inputenc}
\usepackage[T1]{fontenc}

\usepackage{amsfonts}
\usepackage{amsmath}
\usepackage{amssymb}
\usepackage{amsthm}

\usepackage{graphicx}
\usepackage{hyperref}

\newtheorem{thm}{Theorem}
\newtheorem{prop}[thm]{Proposition}

\newtheorem{lem}[thm]{Lemma}
\newtheorem*{conj}{Conjecture}

\theoremstyle{remark}
\newtheorem{rem}[thm]{Remark}

\newcommand{\Z}{\mathbb Z}
\newcommand{\Q}{\mathbb Q}
\newcommand{\F}{\mathbb F}
\newcommand{\C}{\mathbb C}
\newcommand{\q}{\mathfrak q}
\renewcommand{\div}{\operatorname{div}}
\newcommand{\ord}{\operatorname{ord}}
\newcommand{\norm}{\operatorname{N}}
\newcommand{\tr}{\operatorname{Tr}}

\begin{document}

\title{On torsion of superelliptic Jacobians}


\author{Wojciech Wawrów}
\address{Wojciech Wawrów\\
	Adam Mickiewicz University\\
	Faculty of Mathematics and Computer Science\\
	Uniwersytetu Poznańskiego 4\\
	61-614 Poznań, Poland\\\\
	Present address:\\
	London School of Geometry and Number Theory\\
	Department of Mathematics\\
	University College London\\
	Gower Street\\
	London, WC1E 6BT\\
	United Kingdom}
\email{wojtek.wawrow@gmail.com}
\urladdr{https://sites.google.com/view/wojtekwawrow}



\thanks{I thank Prof. Wojciech Gajda for suggesting the topic, Bartosz Naskręcki for help with computational aspects of my research, and Jędrzej Garnek for providing many useful references. I also thank all three of them for many invaluable discussions. Further I would like to thank Sebastian Petersen for his comments on an older version of this paper, as well as Royce Peng for proof-reading the final version. Additional thanks to the anonymous referee for their valuable comments and for pointing me towards additional references.} 


\begin{abstract}
	We prove a result describing the structure of a specific subgroup of the $m$-torsion of the Jacobian of a general superelliptic curve $y^m=F(x)$, generalizing the structure theorem for the $2$-torsion of a hyperelliptic curve. We study existence of torsion on curves of the form $y^q=x^p-x+a$ over finite fields of characteristic $p$. We apply those results to bound from below the Mordell-Weil ranks of Jacobians of certain superelliptic curves over $\Q$.
\end{abstract}

\maketitle

\section{Introduction}

Our objects of study are \emph{superelliptic curves}, defined by equations of the form
\begin{align*}
	C:y^m=F(x)
\end{align*}
for separable polynomials $F$ and $m\geq 2$ not divisible by the characteristic of the base field. This family generalizes hyperelliptic curves, which are curves of the above form for $m=2,\,\deg F>4$, as well as Picard curves, which are the case $m=3,\,\deg F=4$.

We are specifically interested in the points of the form $(\alpha,0)$, where $\alpha$ is a root of $F$. The line $x=\alpha$ intersects $C$ at this point with multiplicity $m$. This lets us find certain divisor classes on the Jacobian $J(C)$, formed by such points and the points at infinity, which are $m$-torsion.

In particular, suppose that $F$ factors as $(x-\alpha_1)\dots(x-\alpha_r)$ in $K$. Consider the group $\Delta$ consisting of classes of divisors of the form
\begin{align*}
	\sum_{i=1}^ra_iR_i-\frac{1}{d}\left(\sum_{i=1}^ra_i\right)\infty,
\end{align*}
where $d=\gcd(m,r)$, $a_i$ are integers whose sum is divisible by $d$, $R_i$ is the point $(\alpha_i,0)$, and $\infty$ is the formal sum of points at infinity of $C$. It is easy to see those classes are $m$-torsion in $J(C)$.

It appears to be a folklore result that for $m=2$, $\Delta$ is the entire $2$-torsion subgroup of $J(C)$. In particular, it is isomorphic to $(\Z/2\Z)^{2g}$, where $g$ is the genus of $C$, equal to either $\frac{r-1}{2}$ or $\frac{r-2}{2}$ according to the parity of $r$. This statement follows for instance from the results of \cite{PS97}, discussed below, when specialized to the case $m=2$.

For $m>2$, those points cannot form all of $m$-torsion for cardinality reasons. For instance, when $\gcd(m,r)=1$, $\Delta$ has order at most $m^{r-1}$, while the theory of abelian varieties tells us that over the algebraic closure of the base field the $m$-torsion consists of $m^{2g}$ points, and $2g>r-1$ as soon as $m>2$.

Describing the entire $m$-torsion of a superelliptic Jacobian is an interesting problem. We provide a result in this direction, showing that $\Delta$ always has the maximal possible size, subject to some obvious relations its points satisfy. Specifically, we have
\begin{prop}
	\label{prop-torsion}
	$\Delta$ is a subgroup of $J(C)$ isomorphic to $(\Z/m\Z)^{r-2}\times\left(\Z/\frac{m}{d}\Z\right)$.
\end{prop}
This result, which does not seem to have appeared in this form in the literature before, can be seen as a generalization of the above mentioned structure theorem for hyperelliptic curves. If $m$ is a prime number the above proposition can be deduced from \cite[Proposition 6.2]{PS97}, which furthermore describes the resulting subgroup as the kernel of a certain endomorphism on the Jacobian of the curve.

We can use the proposition above to find families of curves whose Jacobians have high Mordell-Weil ranks over number fields. As a sample application, we show the following.
\begin{thm}
	\label{thm-app}
	Suppose $p$ is an odd prime, $q$ is a prime factor of $p-1$ and $k$ is an integer divisible by some prime larger than $p$ but not by $p$. Then the Jacobian of the curve
	\begin{align*}
		y^q=x(x-1)\dots(x-(p-1))+k^q
	\end{align*}
	has the Mordell-Weil rank at least $p-1$ over $\Q$.
\end{thm}
The proof of this result follows the methodology of \cite{JT99}, where we use the proposition as a substitute for results about hyperelliptic curves, along with some calculations using Gauss sums over finite fields, similar to ones presented in \cite[\S11]{IR90}.

One of the earliest works where Jacobians of high rank depending on the genus of curves are shown to exist is \cite{N56} where, among other things, N\'eron shows that for any $g$ there are infinitely many curves of genus $g$ with Mordell-Weil rank at least $3g+7$. This work is, however, ineffective. For $g=2$ those bounds have been improved and made effective by Shioda in \cite{S97} where families with Mordell-Weil rank at least $15$ are constructed using a trick introduced by Mestre in \cite{M91}.

The first completely effective result of this kind for curves of arbitrary genus appears be due to Coleman \cite{C85}, where the effective Chabauty method is applied to find families of hyperelliptic curves with Jacobians of high ranks. In \cite{JT99} Coleman's examples are improved, and using a different, more elementary approach the Jacobians are shown to have even higher ranks under additional assumptions. The present work further extends this approach to special families of superelliptic curves, culminating in the explicit examples of Theorem \ref{thm-app} above.

It should be noted that the results mentioned thus far require the genera of the curves, and hence the dimensions of their Jacobians, to grow unboundedly for the ranks to get arbitrarily large. It is a well-known open problem whether the ranks can be arbitrarily large for abelian varieties of a fixed dimension over a fixed number field. For the specific case of elliptic curves over $\Q$, the example with the highest known rank was found by Elkies \cite{Elk} and has rank at least $28$ (exactly $28$ assuming some standard conjectures, see \cite{KSW}). For a recent heuristic argument for boundedness and a brief history of the problem, see \cite{PPVM}.

All this is in stark contrast to the situation over function fields, where curves of arbitrarily high rank have been known for a long time, see \cite{TS67} for the first construction and \cite{U01} for the first construction with nonisotrivial curves. Ulmer has further treated related problems for Jacobians of hyperelliptic curves over function fields. For instance in \cite{U07} he provides examples of families of hyperelliptic Jacobians with arbitrarily high rank, and further verifies the Birch-Swinnerton-Dyer conjecture for them. See \cite{U11} and \cite{U14} for general surveys on elliptic curves and general Jacobians over function fields respectively, including constructions of families with high ranks.

\subsection{Structure of the paper}
Below we recall all the notation, definitions and basic facts used in the following sections. Section \ref{sec-torsion} is devoted to the proof of Proposition \ref{prop-torsion}. In Section \ref{sec-form} we look at the curves of the form $y^q=x^p-x+a$ over a field of characteristic $p$. We show under suitable conditions they have no $q$-torsion over $\F_p$ using methods similar to ones used in \cite{J14}, involving Gauss sums and Hasse-Weil zeta functions. Lastly, in Section \ref{sec-apps}, we use those results and methods based on those in \cite{JT99} to finish the proof of Theorem \ref{thm-app}.

\subsection{Notation and preliminaries}
Let $K$ be an arbitrary perfect field. We define a \emph{superelliptic curve} over $K$ to be a smooth projective model of an affine curve given by an equation of the form
\begin{align*}
	C:y^m=F(x),
\end{align*}
where $m\geq 2$ is an integer not divisible by the characteristic of $K$ and $F\in K[x]$ is separable, i.e. with no repeated factors over $\overline K$. We set $r=\deg F$ and $d=\gcd(m,r)$.

We observe that the affine curve above is smooth, therefore it can be identified with an open subset of its smooth projective model (see \cite[I.6]{H77}). We refer to the points not included in this open subset as the \emph{points at infinity}. Unless necessary, we shall not distinguish between the projective and the affine model of the curve.

Considering the function field of $C$, from \cite[Proposition 3.7.3]{S09} we can show its genus is equal to
\begin{align*}
	g=\frac{1}{2}\big((m-1)(r-1)-(d-1)\big)
\end{align*}
and it has exactly $d$ points at infinity over $\overline K$. We denote their formal sum by $\infty$, which is a divisor of degree $d$ defined over $K$.

We refer to \cite{HS00} for basic facts about Jacobians of curves. We shall identify degree zero divisors with their classes in the Jacobian. We denote the Jacobian of a curve $C$ by $J(C)$, and with $J(C)(K)$ we denote the group of its $K$-rational points.

\section{Proof of Proposition \texorpdfstring{\ref{prop-torsion}}{1}}
\label{sec-torsion}

We may assume that the base field $K$ is algebraically closed. We have the following equalities of divisors on $C$:
\begin{align*}
	\div(x-\alpha_i) &=mR_i-\frac{m}{d}\infty,\\
	\div(y) &=\sum_{i=1}^rR_i-\frac{r}{d}\infty,
\end{align*}
From there it is not hard to see that $\Delta$ is generated as a subgroup of $J(C)$ by the following points:
\begin{align*}
	D_i &=R_i-R_{r-1}\quad\text{for }i=1,\dots,r-2,\\
	D_{r-1} &=dR_{r-1}-\infty.
\end{align*}
and that those points satisfy equalities $mD_i=0$ for $i=1,\dots,r-2$ and $\frac{m}{d}D_{r-1}=0$ in $J(C)$. This gives a surjection from $(\Z/m\Z)^{r-2}\times\left(\Z/\frac{m}{d}\Z\right)$ to $\Delta$ given by
\begin{align*}
	(a_1,\dots,a_{r-1})\mapsto-\sum_{i=1}^{r-1}a_iD_i.
\end{align*}
We wish to show it is also injective.

If the kernel of this map is nontrivial, then, by adding suitable multiples of divisors $\div(x-\alpha_i)$, we can find a principal divisor of the form
\begin{align*}
	D=-\sum_{i=1}^{r-1}a_iR_i+\frac{1}{d}\left(\sum_{i=1}^{r-1}a_i\right)\infty
\end{align*}
with $0\leq a_i<m$ not all zero. We shall prove this is impossible.

Consider the auxiliary divisor
\begin{align*}
	E=-\infty+\sum_{i=1}^{r-1}(m-1)R_i.
\end{align*}
Observe that $\deg E=(r-1)(m-1)-d=2g-1$, therefore Riemann-Roch theorem gives us $\ell(E)=\deg(E)-g+1=g$, where, as usual, $\ell(E)$ is the dimension of the $K$-vector space $L(E)$ of functions $f\in K(C)$ satisfying $E+\div(f)\geq 0$. With this in mind, we can find an explicit basis of $L(E)$.

For any $0<i<r,0<j<m$ let $f_{ij}\in K(C)$ be given by
\begin{align*}
	f_{ij}=\frac{y^j}{\prod_{k\leq i}(x-\alpha_k)}.
\end{align*}
We have
\begin{align*}
	\div(f_{ij})=\sum_{k=1}^i(j-m)R_k+\sum_{k=i+1}^r jR_k+\frac{1}{d}(im-jr)\infty.
\end{align*}
It is clear that $f_{ij}\in L(E)$ iff $im-jr>0$. Let
\begin{align*}
	A &=\{(i,j):0<i<r,0<j<m,im-jr>0\},\\
	B &=\{(i,j):0<i<r,0<j<m,im-jr<0\}.
\end{align*}
The equation $im-jr=0$ has exactly $d-1$ solutions in the range $0<i<r,0<j<m$, which implies that the set $A\cup B$ has $(m-1)(r-1)-(d-1)=2g$ elements. Further, since $(r-i)m-(m-j)r=-(im-jr)$, the map $(i,j)\mapsto(r-i,m-j)$ gives a bijection from $A$ to $B$, showing $A$ has exactly $\frac{1}{2}|A\cup B|=g$ elements.

It follows that $\{f_{ij}:(i,j)\in A\}$ is a set of $g=\ell(E)$ elements of $L(E)$, so to show it is a basis it is enough to show their linear independence. Assume there exist $b_{ij}\in K$, not all zero, such that
\begin{align*}
	\sum_{(i,j)\in A}b_{ij}f_{ij}=0.
\end{align*}
Take the largest index $k$ such that $b_{kj}\neq 0$ for some $j$. Observe $f_{ij}$ for $i<k$ are all regular at $R_k$, while $f_{kj}$ has a pole of order $m-j$ at this point. Thus, letting $l$ be the least index such that $b_{kl}\neq 0$, the left-hand side above has a pole of order $m-l$ at $R_k$, so clearly is not zero. This contradiction establishes linear independence, and hence that the set above is a basis.

Consider again our divisor $D$. Suppose it is principal, that is, there is a nonzero $f\in K(C)$ such that $\div(f)=D$. We have
\begin{align*}
	E+\div(f)=E+D=\sum_{i=1}^{r-1}(m-1-a_i)R_i+\left(-1+\sum_{i=1}^{r-1}a_i\right)\infty\geq 0
\end{align*}
since the $a_i$ are all at most $m-1$ and their sum is positive. This means $f\in L(E)$, so that it can be written in terms of the basis we have found:
\begin{align*}
	f=\sum_{(i,j)\in A}c_{ij}f_{ij}
\end{align*}
with $c_{ij}\in K$. But each $f_{ij}$ has a zero at $R_r$, hence so does $f$, which is clearly not the case since the coefficient of $R_r$ in $D=\div(f)$ is zero. We conclude $D$ is not a principal divisor, as we wanted.\qed

\begin{rem}
	\label{rem-r1}
	When $d=1$ we have that $\infty$ is just a single point, and $\Delta$ is generated by the points $D_i'=R_i-\infty$ for $i=1,\dots,r-1$. The proposition then shows that an integer linear combination of those $D_i'$ is zero in $J(C)$ if and only if all of its coefficients are divisible by $m$.
\end{rem}

\section{Curves of the form \texorpdfstring{$y^q=x^p-x+a$}{y\^{}q=x\^{}p-x+a}}
\label{sec-form}

We move on to study superelliptic curves $C$ with equations of the form
\begin{align*}
	C:y^q=x^p-x+a
\end{align*}
over finite fields of characteristic $p$, where $p,q$ are distinct primes and $a\in\F_p^\times$. Observe that the polynomial on the right-hand side is always separable, since its derivative $-1$ doesn't vanish. We denote the Jacobian of $C$ by $J$.

Assume first $q\mid p-1$. We shall compute the zeta function of this curve using Gauss sums of additive and multiplicative characters.

For any additive character $\psi:\F_p\to\C^\times$ we define a character $\psi_n:\F_{p^n}\to\C^\times$ by
\begin{align*}
	\psi_n(\alpha)=\psi(\tr_{\F_{p^n}/\F_p}(\alpha)),
\end{align*}
where $\tr_{\F_{p^n}/\F_p}$ denotes the field-theoretic trace map from $\F_{p^n}$ to $\F_p$. Similarly, for a multiplicative character $\chi:\F_p\to\C$, we define $\chi_n:\F_{p^n}\to\C$ by
\begin{align*}
	\chi_n(\alpha)=\chi(\norm_{\F_{p^n}/\F_p}(\alpha)),
\end{align*}
where $\norm_{\F_{p^n}/\F_p}$ is the field-theoretic norm.
\begin{rem}
	We adopt the convention for multiplicative characters $\chi$ that $\chi(0)=0$ for nontrivial characters $\chi$, but $\chi(0)=1$ for $\chi$ the trivial character.
\end{rem}

Since $\gcd(p,q)=1$, $C$ has exactly one point at infinity, so we just need to count the points on its affine part.
\begin{lem}
	$C$ has exactly
	\begin{align*}
		\sum_{w-z=a}\sum_\psi\sum_\chi\psi_n(z)\chi_n(w)
	\end{align*}
	affine points over $\F_{p^n}$, where the first sum ranges over all pairs $w,z\in\F_{p^n}$ satisfying $w-z=a$, the second over all additive characters of $\F_p$ and the third one over multiplicative characters of $\F_p$ of order dividing $q$.
\end{lem}
\begin{proof}
	For each $z,w\in\F_{p^n}$ we count the number of points on $C$ satisfying $x^p-x=z$ and $y^q=w$. For such a solution to exist we need to have $w-z=a$, so we only need to consider pairs $z,w$ satisfying this. For each such pair it is sufficient to count the number of solutions $x,y$ to $x^p-x=z,y^q=w$.
	
	It is standard that for any $z,w$ the number of solutions to $y^q=w$ in $\F_{p^n}$ is equal to 
	\begin{align*}
		\sum_\chi\chi_n(w),
	\end{align*}
	while the number of solutions to $x^p-x=z$ is equal to
	\begin{align*}
		\sum_\psi\psi_n(z),
	\end{align*}
	where the ranges of the sums are as in the statement of the lemma. Combining all of those observations we get the formula for the number of points in $C(\F_{p^n})$.
\end{proof}

We consider modified Gauss sums defined by
\begin{align*}
	G_a(\psi_n,\chi_n)=\sum_{w-z=a}\psi_n(z)\chi_n(w)
\end{align*}
with the sum over $z,w\in\F_{p^n}$ satisfying $z-w=a$. The previous lemma gives
\begin{align*}
	|C(\F_{p^n})|=1+\sum_\psi\sum_\chi G_a(\psi_n,\chi_n).
\end{align*}
The sums $G_a$ have the following properties:
\begin{lem}
	\label{lem-Hasse-Davenport}
	Let $\psi$ be an additive character of $\F_p$ and $\chi$ a multiplicative character of $\F_p$.
	\begin{itemize}
		\item If $\psi,\chi$ are both trivial, then $G_a(\psi_n,\chi_n)=p^n$.
		\item If exactly one of $\psi,\chi$ is trivial, then $G_a(\psi_n,\chi_n)=0$.
		\item If both $\psi,\chi$ are nontrivial, then $-G_a(\psi_n,\chi_n)=(-G_a(\psi,\chi))^n$.
	\end{itemize}
\end{lem}
\begin{proof}
	The first two claims are immediate. Observe
	\begin{align*}
		G_a(\psi_n,\chi_n) &=\sum_{w-z=a}\psi_n(z)\chi_n(w)=\sum_w\psi_n(w-a)\chi_n(w)\\
		&=\psi_n(-a)\sum_w\psi_n(w)\chi_n(w)=\psi(-a)^nG(\psi_n,\chi_n),
	\end{align*}
	where $G(\psi_n,\chi_n)$ is the usual Gauss sum. The last statement now follows from the usual Hasse-Davenport relations for Gauss sums, see \cite[\S11.4]{IR90}.
\end{proof}
We therefore have
\begin{align*}
	|C(\F_{p^n})|=1+p^n-\sum_{\psi\neq 1}\sum_{\chi\neq 1}(-G_a(\psi,\chi))^n,
\end{align*}
where with $1$ we denote the trivial character (both additive and multiplicative). This gives us an explicit formula for the zeta function of $C$:
\begin{align*}
	Z(C,T)=\frac{\prod_{\psi\neq 1}\prod_{\chi\neq 1}(1+G_a(\psi,\chi)T)}{(1-T)(1-pT)}.
\end{align*}
By \cite[Section 5.4]{P06}, evaluating the numerator at $T=1$ gives us the number of points on the Jacobian of $C$ over $\F_p$, i.e.
\begin{align*}
	|J(\F_p)|=\prod_{\psi\neq 1}\prod_{\chi\neq 1}(1+G_a(\psi,\chi)),
\end{align*}

With this formula we can now establish
\begin{prop}
	\label{thm-q-mid-p-1}
	If $q\mid p-1$, the Jacobian of $C$ has no $q$-torsion over $\F_p$.
\end{prop}
\begin{proof}
	Let $\zeta_{pq}$ be a primitive $pq$-th root of unity. Note that $G_a(\psi,\chi)\in\Z[\zeta_{pq}]$ for all $a,\psi,\chi$, and that $\Z[\zeta_{pq}]$ is a Dedekind domain. Take any prime ideal factor $\q$ of $q$ in this ring.
	
	We have a classical equality of ideals $(q)=(1-\zeta_q)^{q-1}$ in any ring containing $\zeta_q$, so that necessarily $1-\zeta_q\in\q$. It follows  we have $\chi_n(w)\equiv 1\pmod\q$ for all $w\in\F_{p^n}^\times$, hence for $\psi,\chi\neq 1$ we get
	\begin{align*}
		G_a(\psi,\chi) &\equiv\sum_{z\neq-a}\psi(z)=-\psi(-a)+\sum_z\psi(z)=-\psi(-a)\pmod\q,
	\end{align*}
	thus
	\begin{align*}
		|J(\F_p)| &\equiv\prod_{\psi\neq 1}\prod_{\chi\neq 1}(1-\psi(-a))\pmod\q.
	\end{align*}
	Observe that since $-a\in\F_p^\times$, $\psi(-a)$ is a primitive $p$-th root of unity for every $\psi\neq 1$, so $1-\psi(-a)$ is a factor of $p$ in $\Z[\zeta_{pq}]$. It follows that $1-\psi(-a)\not\in\q$, as otherwise we would find $p\in\q$, which cannot hold as $p,q$ are distinct rational primes. It follows $1-\psi(-a)$ is not in $\q$, hence neither is $|J(\F_p)|$. Consequently $q\nmid |J(\F_p)|$, so $J(\F_p)$ has no $q$-torsion.
\end{proof}

If we now drop the condition that $q$ divides $p-1$ and merely assert that it is different from $p$, we can instead reason using a field extension $\F_{p^k}$ for $k$ such that $q\mid p^k-1$. This way we can give an exact condition for when $q$-torsion exists in $\F_p$:
\begin{thm}
	Let $p,q$ be distinct primes and let $k=\ord_qp$ be the least $k$ such that $q\mid p^k-1$. Then the Jacobian of $C$ has no $q$-torsion over $\F_p$ iff $p\nmid k$.
\end{thm}
\begin{proof}
	We can repeat the reasoning preceding Theorem \ref{thm-q-mid-p-1} and in the proof of the theorem to find
	\begin{align*}
		|J(\F_{p^k})|=\prod_{\psi\neq 1}\prod_{\chi\neq 1}(1-(-G_a(\psi_k,\chi)))\equiv\prod_{\psi\neq 1}\prod_{\chi\neq 1}(1-\psi_k(-a))\pmod\q,
	\end{align*}
	where this time $\psi$ ranges over all additive characters of $\F_p$, while $\chi$ ranges over all multiplicative characters of $\F_{p^k}$ of order dividing $q$. It is clear that $\psi_k(-a)=\psi(-a)^k$, so as long as $p\nmid k$ it is again a primitive root of unity. Hence $J(\F_{p^k})$ has no $q$-torsion, thus neither does its subgroup $J(\F_p)$.
	
	When $p\mid k$ we have $\psi_k(-a)=\psi(-a)^k=1$, so that $|J(\F_{p^k})|\equiv 0\pmod\q$, hence $q\mid |J(\F_{p^k})|$. From the following lemma we also have $q\mid |J(\F_p)|$, so $J(\F_p)$ contains a $q$-torsion point.
\end{proof}
\begin{lem}
	For $p,q,k$ as above we have
	\begin{align*}
		|J(\F_{p^k})|=|J(\F_p)|^k.
	\end{align*}
\end{lem}
\begin{proof}
	From the Weil conjectures it follows there is a unique (up to ordering) collection of $2g$ numbers $\alpha_1,\dots,\alpha_{2g}$ such that, for all $n$,
	\begin{align*}
		p^n+1-C(\F_{p^n})=\sum_{i=1}^{2g}\alpha_i^n.
	\end{align*}
	But for $k\nmid n$ we have $q$ and $p^n-1$ relatively prime, so that $y\mapsto y^q$ is a bijection on $\F_{p^n}$. It follows that $C(\F_{p^n})=p^n+1$, hence
	\begin{align}
		\label{eq-zero}
		\sum_{i=1}^{2g}\alpha_i^n=0.
	\end{align}
	Note that if we multiply each of $\alpha_1,\dots,\alpha_{2g}$ by $\zeta_k^j$, where $\zeta_k$ is a primitive $k$-th root of unity and $j$ is arbitrary, we get
	\begin{align*}
		\sum_{i=1}^{2g}(\zeta_k^j\alpha_i)^n=\zeta_k^{nj}\sum_{i=1}^{2g}\alpha_i^n=\sum_{i=1}^{2g}\alpha_i^n
	\end{align*}
	for all $n$---if $k\nmid n$ this follows from \eqref{eq-zero}, while if $k\mid n$ it is immediate. The uniqueness statement above leads us to a conclusion that $\alpha_i\mapsto\zeta_k^j\alpha_i$ is a permutation for any $j$. Using again \cite[Section 5.4]{P06},
	\begin{align*}
		|J(\F_p)|^k &=\left(\prod_{i=1}^{2g}(1-\alpha_i)\right)^k=\prod_{i=1}^{2g}\prod_{j=0}^{k-1}(1-\zeta_k^j\alpha_i)=\prod_{i=1}^{2g}(1-\alpha_i^k)=|J(\F_{p^k})|.\qedhere
	\end{align*}
\end{proof}

\begin{rem}
	As pointed out by a referee, essentially the same argument shows that $|J(\F_{p^{k'}})|=|J(\F_p)|^{k'}$ for any $k'\mid k$. The lemma holds more generally for any superelliptic curve $y^q=F(x)$ with $F\in\F_p[x]$ of degree not divisible by $q$. Further, numerical evidence suggests that the following holds:
\end{rem}
\begin{conj}
	Consider a superelliptic curve $C$ given by $y^q=F(x)$ with $q$ prime, $F\in\F_p[x]$ of degree not divisible by $q$ and set $k=\ord_qp$. There exists an isomorphism
	\begin{align*}
		J(\F_{p^k})\cong J(\F_p)^k
	\end{align*}
	of abstract groups.
\end{conj}

\begin{rem}
	The previous two theorems still hold when we take $C$ to be defined by an equation $y^{q^l}=x^p-x+a$. The proofs are analogous using the fact $1-\zeta_{q^l}\in\q$ for any prime $\q$ of $\Z[\zeta_{pq^l}]$ containing $q$. We get that $q\mid|J(\F_p)|$ iff $p\mid\ord_qp$.
\end{rem}

\section{Applications to bounding ranks}
\label{sec-apps}

We now apply the established results to reductions of certain curves over the rationals to bound their Mordell-Weil ranks from below. Specifically, consider curves defined by
\begin{align*}
	C:y^m=(x-a_1)\dots(x-a_r)+k^m,
\end{align*}
where $a_1,\dots,a_r$ and $k$ are integers and $m,r$ are relatively prime. We take the points $P_i=(a_i,k)\in C$, their images $D_i=P_i-\infty\in J(C)$ under the Albanese map and the subgroup $\Gamma$ those images generate. It is clearly a subgroup of $J(C)(\Q)$. The following result is implicit in \cite{JT99} for $m=2$.
\begin{prop}
	\label{prop-application}
	Assume there is a prime $p$ such that:
	\begin{enumerate}
		\item $m$ is not divisible by $p$,
		\item $k$ is divisible by $p$,
		\item the $a_i$ are pairwise incongruent modulo $p$.
	\end{enumerate}
	Suppose further $\Gamma$ contains no nontrivial $m$-torsion. Then $\Gamma$ is free of rank $r-1$.
\end{prop}
\begin{proof}
	The sum of all $D_i$ is equal to $\div(y-k)$, so $\Gamma$ is generated by $D_1,\dots,D_{r-1}$. It is enough to show there is no relation between those points.
	
	Since $p\mid k$, the reduction of the equation of $C$ modulo $p$ is
	\begin{align*}
		\widetilde C:y^m=(x-a_1)\dots(x-a_r).
	\end{align*}
	We have $p\nmid m$ by the first assumption, and the right-hand side is separable by the third, from which we deduce $C$ has good reduction modulo $p$. It follows $J(C)$ also has good reduction (see \cite[Corollary VII.12.3]{CS86}), which induces a reduction homomorphism $J(C)(\Q)\to J(\widetilde C)(\F_p)$.
	
	Suppose there is a nontrivial relation between the points $D_1,\dots,D_{r-1}$, say
	\begin{align*}
		\sum_{i=1}^{r-1}a_iD_i=0
	\end{align*}
	with not all $a_i$ zero. Reordering the points and changing the sign if necessary, we may assume $a_1>0$ and $a_1$ is the smallest possible among all such relations. The relation is preserved by the reduction homomorphism, so
	\begin{align*}
		\sum_{i=1}^{r-1}a_i\widetilde D_i=0,
	\end{align*}
	where $\widetilde D_i$ is the image of $D_i$ under reduction. But the reduction of the point $P_i=(a_i,k)$ is $(\widetilde{a_i},0)$ (as we assumed $p\mid k$), so the points $\widetilde D_i$ coincide with the points $D_i'$ considered in Remark \ref{rem-r1}. This gives rise to a vanishing linear combination of the $D_i'$, which by the remark implies the coefficients are all divisible by $m$. Writing $a_i=mb_i$ we have $b_1<a_1$, so, by minimality of $a_1$, the point $D=\sum_{i=1}^{r-1}b_iD_i\in J(C)(\Q)$ is nonzero. However, the original relation gives us $mD=0$, which contradicts the assumption that $\Gamma$ has no $m$-torsion.
\end{proof}
\begin{rem}
	Suitable versions of this proposition also hold for curves satisfying $\gcd(m,r)>1$, as well as for curves over arbitrary number fields in place of $\Q$, and lead to more general variants of Theorem \ref{thm-app}
\end{rem}

\begin{proof}[Proof of Theorem \ref{thm-app}]
	Let $C$ be a curve as in the statement of the theorem. First we show $C$ has no $q$-torsion over $\Q$. The reduction of the equation of this curve modulo $p$ is 
	\begin{align*}
		\widetilde C:y^q=x^p-x+k^q
	\end{align*}
	and, as we have noted before, the right-hand side is separable, so as before the curve and its Jacobian have good reduction and we get a reduction homomorphism $J(C)(\Q_p)\to J(\widetilde C)(\F_p)$.
	
	The kernel of this reduction homomorphism is isomorphic to a group associated to a $g$-parameter formal group over $\Z_p$ (see \cite[\S C.2]{HS00}). By \cite[Theorems II.9.3 and II.9.4]{S65} this group has no torsion, so all torsion of $J(C)(\Q_p)$ survives into $J(\widetilde C)(\F_p)$. But Theorem \ref{thm-q-mid-p-1} tells us this last group has no $q$-torsion, therefore neither does $J(C)(\Q_p)$ nor $J(C)(\Q)$. We are now done by the previous proposition.
\end{proof}
Using Dirichlet's theorem on primes in arithmetic progressions, this gives us, for any fixed $q$, families of superelliptic curves whose Jacobians have Mordell-Weil ranks that grow at least linearly with the genus of the curve.

\begin{rem}
	The methods described above can be applied to curves which have the form $y^q=x^p+a$ after reduction modulo some prime $\ell\equiv 1\pmod{pq}$. Torsion on Jacobians of such curves is studied at length in \cite{J14}, see for instance Lemma 4 of that paper, where one can find a different proof of the analogue of Theorem \ref{thm-q-mid-p-1} for the case $q=3$.
\end{rem}

\bibliographystyle{plain}
\bibliography{Torsion}

\end{document}